\documentclass[11pt,a4paper]{article}
\usepackage[utf8x]{inputenc}
\usepackage{ucs}
\usepackage{amsmath}
\usepackage{amsfonts}
\usepackage{amssymb}
\usepackage{amsthm}
\usepackage{graphicx}
\usepackage{hyperref}
\usepackage{bera}
\usepackage[width=17.00cm, height=26.50cm, left=2.00cm, top=1.00cm]{geometry}

\usepackage{color}

\newtheorem{theorem}{Theorem}
\newtheorem{lemma}[theorem]{Lemma}
\newtheorem{proposition}[theorem]{Proposition}
\newtheorem{corollary}[theorem]{Corollary}

\newcommand{\hp}{\mathcal{H}_{p} (E)}
\newcommand{\hpr}{\mathcal{H}_{p}^{\mathrm{rad}} (E)}
\newcommand{\hprm}{\mathcal{H}_{p,m}^{\mathrm{rad}} (E)}
\newcommand{\spr}{S_{p}^{\mathrm{rad}} (E)}
\newcommand{\sr}{\sigma^{\mathrm{rad}}}
\newcommand{\srh}{\sigma^{\mathrm{rad}}_{\mathcal{H}_{p}(E)}}

 \newcommand{\bj}{\mathbf j}

\DeclareMathOperator{\re}{Re}
\DeclareMathOperator{\Rad}{Rad}

\title{Almost sure-sign convergence of Hardy-type Dirichlet series}

\author{Daniel Carando\footnote{Departamento de Matematica - Pab I,
Facultad de Cs. Exactas y Naturales, Universidad de Buenos Aires, 1428 Buenos Aires (Argentina) and IMAS - CONICET (Argentina). Partially supported by CONICET-PIP 0624, PICT 2011-1456 and UBACyT 20020130100474BA}
\and
Andreas Defant\footnote{Institut f\"ur Mathematik. Universit\"at Oldenburg. D-26111 Oldenburg (Germany). Supported by MICINN  MTM2011-22417 }
\and
Pablo Sevilla-Peris \footnote{Instituto Universitario de Matem\'atica Pura y Aplicada. Universitat Polit\`ecnica de Val\`encia. 46022 Valencia (Spain). Supported by MICINN  MTM2011-22417 and UPV-SP201207000 }}

\date{}

\begin{document}

\maketitle

\begin{abstract}
\noindent Hartman proved in 1939 that the width of the largest  possible strip in the complex plane, on which a Dirichlet series $\sum_n a_n n^{-s}$ is  uniformly a.s.-sign convergent (i.e.,  $\sum_n \varepsilon_n a_n n^{-s}$ converges uniformly for almost all  sequences of signs $\varepsilon_n =\pm 1$) but does not  convergent absolutely, equals $1/2$. We study this result from a more modern point of view
 within the framework of so-called Hardy-type Dirichlet series with values in a Banach space
\end{abstract}


\section{Introduction}

The natural domains of convergence of Dirichlet series are half-plains. Given a  Dirichlet series $D= \sum a_{n} n^{-s}$ there are  three abscissas which define the biggest half-plains on which $D$ converges, converges uniformly and converges absolutely:
\begin{equation} \label{Bohrabs}
\sigma_{c}(D) \leq \sigma_{u} (D) \leq \sigma_{a} (D) \, .
\end{equation}
Whereas it is not difficult to show that
 \begin{equation} \label{Bo}
 \sup_{D \text{ Dir. ser.}} \sigma_{a} (D) - \sigma_{u} (D)  =1  \,,
\end{equation}
 the  main issue in the 1910's was to decide what  the maximal width of the band on which a Dirichlet series can converge uniformly but not absolutely is. This problem was first considered by Harald Bohr and consists on computing the following number
\[
S : = \sup_{D \text{ Dir. ser.}} \sigma_{a} (D) - \sigma_{u} (D) \,.
\]
Bohr himself \cite{Bo13_Goett} showed in 1913 that $S \leq 1/2$\,,
 but it was not until 1931 that Bohnenblust and Hille \cite{BoHi31} proved that indeed
 \begin{equation} \label{BoBoHi}
S = 1/2\,.
\end{equation}
   The proof of the lower bound for $S$ by Bohnenblust and Hille is long and
involved. A few years  later Hartman gave in \cite{Ha39} a different proof for the lower bound, based on  probabilistic arguments. Let us be more precise. On $\{-1,1\}$ we consider the  probability $\mathbf{P}(-1)=\mathbf{P}(1)=1/2$ and on
$\{ -1 , 1\}^{\mathbb{N}}$ its product probability. From now on $(\varepsilon_{n})_{n}$ will always be a sequence of signs in $\{ -1 , 1\}^{\mathbb{N}}$.
We say  that $D= \sum a_{n} n^{-s}$ is (uniformly) a.s.-sign  convergent on a half plane
$[\re > \sigma]$ whenever $D= \sum \varepsilon_n a_{n} n^{-s}$ (uniformly) converges  on $[\re > \sigma]$ outside of a zero set of signs
$\varepsilon_n$.

Given a Dirichlet series $D= \sum a_{n} n^{-s}$, Hartman in \cite{Ha39}  (with a slightly different notation) considers the following abscissas
\begin{gather*}
\sr_{c} (D) := \inf \big\{ \sigma \in \mathbb{R} \colon\textstyle \sum a_{n} n^{-s} \text{ a.s.-sign convergent on } [\re > \sigma] \big\} \\
\sr_{u} (D) := \inf \big\{ \sigma \in \mathbb{R}\colon\textstyle \sum  a_{n} n^{-s} \text{ uniformly a.s.-sign  convergent on } [\re > \sigma] \big\}\,
\end{gather*}
(obviously, it doesn't make any sense to define an analogous notion like $\sr_{a} (D)$). In general the abscissas $\sr_{c}  (D)$
and $\sr_{u}  (D)$  are different from $\sigma_{c} (D)$
and $\sigma_{u} (D)$, respectively.
The two main result in Hartman's article are (compare with \eqref{Bo} and \eqref{BoBoHi})
\begin{align} \label{Ha1}
\sup_{D \text{ Dir. ser.}} \sigma_{a} (D) -  \sr_{c} (D)
= \frac{1}{2}
\end{align}
 and
\begin{align} \label{Ha2}
 \sup_{D \text{ Dir. ser.}} \sigma_{a} (D) -  \sr_{u}  (D) = \frac{1}{2}\,.
\end{align}
 Hartman in particular proves that $ \sup_{D} \sigma_{a} (D) -  \sr_{u}  (D)
\leq  \sup_{D} \sigma_{a} (D) -  \sigma_{u}  (D)$,  and estimating the first $\sup$ from below  he   produces a substantially different proof of the lower bound of $S$. His proof is probabilistic and goes through almost periodic functions with random Fourier coefficients.

Recently, many authors have shown  new interest in the Bohr-Bohnenblust-Hille circle of ideas
(see the recent monograph \cite{QuQu13} and also \cite{AlOlSa12,BaCaQu06,Ba02,BaPeSe14,BaQuSe15,BaDeFrMaSe14,CaDeSe14,Br08,
DeGaMaPG08,DeFrOrOuSe11,DeMaSc12,DeScSe14, HeLiSe97,  KoQu01, MaQu10,Qu95}), and probabilistic arguments have shown to be of great interest in this theory. Being Hartman's paper the first time when
such probabilistic arguments were used to deal with Dirichlet series,
our aim in this note is to look at his results from this more modern and  more general point of view that we believe clarifies the original argument.

\paragraph{Vector-valued Hardy-type Dirichlet series}
We are going to work with Hardy spaces of Dirichlet series with values in a Banach space. In this way we continue and extend our work from \cite{CaDeSe14}.
Given a Banach space $E$, we consider the one-to one correspondence  between the spaces  $\mathfrak{P}(E)$ (all formal power series $\sum_\alpha c_{\alpha} z^{\alpha}$ in infinitely many variables with coefficients $c_\alpha \in E$) and  $\mathfrak{D}(E)$
(all formal Dirichlet  series $\sum_n a_n \frac{1}{n^s}$ with coefficients $a_n \in E$)
\begin{equation} \label{willybrandt}
\mathfrak{P}(E) \longrightarrow \mathfrak{D}(E)
\,\,, \,\,\,\sum c_{\alpha} z^{\alpha}
\mapsto
 \sum a_{n} n^{-s}
 \end{equation}
given by $a_{n} = c_{\alpha }$ if
$n = p^{\alpha} = p_{1}^{\alpha_{1}} \cdots p_{k}^{\alpha_{k}}$,
where $p_{1}< p_{2} < p_{3} < \ldots$ stands for the sequence of prime numbers (see \cite{CaDeSe14}).

Let us recall the definition of  Hardy spaces  of $E$-valued Dirichlet series (first defined by Bayart for $E=\mathbb{C}$ in \cite{Ba02} and later for arbitrary $E$  in \cite{CaDeSe14}).
 For every $f \in L_{1} (\mathbb{T}^{\mathbb{N}};E)$ (the Banach space of Bochner integrable $E$-valued functions defined on the infinite dimensional torus $\mathbb{T}^{\mathbb{N}}$ with  the normalized Lebesgue measure $dz$) and every multi index $\alpha \in \mathbb{Z}^{(\mathbb{N})}$ (all finite sequences $\alpha = (\alpha_n)_{n \in \mathbb{Z}}$) we as usual  denote the $\alpha$th  Fourier coefficient of $f$ by
$
\hat{f} (\alpha) = \int_{\mathbb{T}^{\mathbb{N}} } f(z) z^{-\alpha} dz\,.
$
Now define for $1 \leq p < \infty$ the Hardy space
$$
H_{p} (\mathbb{T}^{\mathbb{N}};E) := \big\{ f \in L_{p} (\mathbb{T}^{\mathbb{N}};E) \colon \hat{f} (\alpha) \neq 0 \text{ only if } \alpha \in \mathbb{N}_{0}^{(\mathbb{N})} \big\}
$$
(with the norm induced by $L_{p} (\mathbb{T}^{\mathbb{N}};E)$) and let  $$\hp$$ by definition  be the image of the Banach space $H_{p} (\mathbb{T}^{\mathbb{N}};E)$ by the aforementioned correspondence~\eqref{willybrandt} (again with the norm coming from $H_{p} (\mathbb{T}^{\mathbb{N}};E)$).
We also consider
$$\mathcal{H}_{\infty}(E)\,,$$
 the space of $E$-valued Dirichlet series such that $\sum_{n} a_{n} \frac{1}{n^{s}}$ defines a bounded, holomorphic function on $[\re s > 0]$, with the norm $\Vert \sum a_{n} n^{-s} \Vert_{\mathcal{H}_{\infty} (E)} := \sup_{\re s >0} \Vert \sum_{n} a_{n} \frac{1}{n^{s}} \Vert_{E}$. We note that this Banach space through the identification in \eqref{willybrandt} coincides isometrically with $H_{\infty} (\mathbb{T}^{\mathbb{N}};E)$ if and only if
 $E$ has the analytic Radon-Nikodym property (see \cite{DePe}). In the scalar case we abbreviate
 \[
 \mathcal{H}_p = \mathcal{H}_{p}(\mathbb{C})\,,1 \leq p \leq \infty\,.
 \]
 Clearly, we have that
\begin{equation} \label{HeLiSe}
\mathcal{H}_2 = \Big\{ \textstyle \sum a_{n} n^{-s} \colon \|D\|_{\mathcal{H}_2}=\Big( \displaystyle\sum_{n=1}^{\infty} \vert a_{n} \vert^{2} \Big)^{\frac{1}{2}}< \infty \Big\} \,,
\end{equation}
 a Hilbert space intensively studied by  Hedenmalm, Lindqvist and Seip in \cite{HeLiSe97}.

\paragraph{State of art.}
  For any Banach space $E$ and any $1 \leq p \leq \infty$ define the  $\mathcal{H}_{p}(E)$-abscissa of a Dirichlet series $\sum a_n n^{-s}$ by
\[
\sigma_{\mathcal{H}_{p}(E)} (D) :  = \inf \big\{ \sigma \in \mathbb{R}\colon \textstyle\sum \frac{a_{n}}{n^{\sigma}} n^{-s} \in \hp  \big\} \,.
\]
  Then, given $1 \leq p \leq \infty$,  we have (for $p=\infty$ see \cite{BoHi31} (scalar case) and
 \cite[Theorem~1]{DeGaMaPG08} (vector-valued  case),  and  for $1 \leq p < \infty$ see \cite{BaCaQu06,Ba02}  (scalar case) and \cite{CaDeSe14} (vector-valued  case))
\begin{equation} \label{pariserplatz}
S_{p}(E) :=  \sup_{D \in \mathfrak{D}(E)}  \sigma_{a}(D) - \sigma_{\mathcal{H}_{p}(E)} = 1 - \frac{1}{\cot E} \,;
\end{equation}
recall that $E$ has cotype $q$ (with $2 \leq q < \infty$) if there is a constant $C>0$ such that for every finite choice of elements $x_{1}, \ldots , x_{N} \in E$ we have
$
\big( \sum_k \Vert x_{k} \Vert_{E}^{q}  \big)^{\frac{1}{q}}
\leq C \big( \int_{\mathbb{T}^{N}} \Vert \sum_k x_{k} z_{k} \Vert_{E}^{2} dz  \big)^{1/2}\,,
$
and
$$\cot E := \inf \Big\{ q \in [2, \infty[ \, \colon \,E \text{ has cotype } q \Big\}\,.$$
Let us comment on the special case $E= \mathbb{C}$ for which  $\cot \mathbb{C} =2$. For  $p=\infty$ we know by Bohr's fundamental theorem from \cite{Bo13_Goett} (see also \cite[Theorem~6.2.3]{QuQu13})
that
\begin{equation} \label{Bohrfun}
\sigma_u(D)=\sigma_{\mathcal{H}_\infty}(D)\,,
\end{equation}
 hence \eqref{pariserplatz} implies \eqref{BoBoHi}. For $p=2$ we have
 \begin{equation} \label{Kinfun}
\sr_c(D)=\sigma_{\mathcal{H}_2}(D)\,,
\end{equation}
so in this case \eqref{pariserplatz} implies \eqref{Ha1}. Indeed, by Khinchin's inequality it is well-known that a  scalar sequence $x=(x_n)$ is a.s.-sign summable (i.e., $\sum_n \varepsilon_n x_n$ converges for almost all possible choices of signs $\varepsilon_n$) if and only if  $x \in \ell_2$.
This, together with  \eqref{HeLiSe}, is what we need. We will come back to this issue later.\\

Is it also possible to recover \eqref{Ha2} within the setting of Hardy-type Dirichlet series?  Given $1 \leq p \leq \infty$ and a Banach space $E$, we  define what is going to be one of our our main objects,
\begin{equation} \label{ostbahnhof}
\hpr := \big\{ \textstyle\sum a_{n} n^{-s} \colon \forall \text{ a.e. } \varepsilon_{n} = \pm 1 \,,\,\,  \sum \varepsilon_{n} a_{n} n^{-s} \in \hp \big\} \,.
\end{equation}
Then, for a given Dirichlet series $D \in \mathfrak{D}(E)$, we consider the  abscissa
\[
\srh (D)  := \inf \big\{ \sigma \in \mathbb{R} \colon \textstyle\sum \frac{a_{n}}{n^{\sigma}} n^{-s} \in \hpr  \big\} \,,
\]
and again the aim is to determine  the maximal distance between $\sigma_a(D)$ and $\srh (D)$.
\paragraph{Summary.} Our first main result (Theorem~\ref{alexanderplatz1}) is a proper extension of Hartman's main result from \eqref{Ha2} and an analogue of \eqref{pariserplatz} in the setting of a.s.-sign convergence of Hardy-type Dirichlet series: For every Banach space $E$ and $1 \leq p \leq \infty$
\begin{equation}\label{alexanderplatz}
\spr :=  \sup_{D \in \mathfrak{D}(E)}  \sigma_{a}(D) - \srh (D)  \,=\, 1 - \frac{1}{\cot E} \,.
\end{equation}
Indeed, \eqref{alexanderplatz} recovers  \eqref{Ha2}  since  $\cot \mathbb{C} =2$  and
 $\sr_{\mathcal{H}_{\infty}}(D)$ is the abscissa $\sr_{u}(D)$ defined by Hartman.
 Moreover, we show that $\spr \leq S_p(E)$  (Corollary~\ref{friedrichstr}), hence \eqref{alexanderplatz} also recovers
 \eqref{pariserplatz}.
For the proof of \eqref{alexanderplatz} we distinguish between finite and infinite dimensional Banach spaces $E$. In the rest of our article we graduate \eqref{alexanderplatz}.
Following an idea from \cite{BoHi31}, we give precise estimates for the  $m$th graduation of
$\spr$ along $m$-homogeneous Dirichlet series (see Proposition~\ref{contrast} and Proposition~\ref{17juni}).
In the scalar case, we graduate   \eqref{alexanderplatz} along the length of the considered Dirichlet series; here our main results are Theorem~\ref{main} and Theorem~\ref{we finish}. Finally, in the Appendix we
show the remarkable fact that the maximum width of the strips of a.s.-sign but not absolute convergence and of uniform a.s.-sign but not absolute convergence coincide (see \eqref{Ha1} and \eqref{Ha2}) extends to the vector-valued case.

\section{The Banach space $\boldsymbol{\hpr}$}
In this section we collect a few facts on $\hpr$ needed later.
First of all, we need a norm on $\hpr$. We denote by $(r_{n})_{n}$ the system of Rademacher functions on $[0,1]$. We are going to use the following, fundamental for us,
fact (see e.g. \cite[Theorem~12.3]{DiJaTo95}): Given a sequence $(x_{n})_{n}$ in a Banach space $X$, the series $\sum_{n} r_{n} x_{n}$ converges almost surely if and only if $\sum_{n} r_{n} x_{n}$ converges in $L_{p} ([0,1];X)$ for some (and then all)
$0<p<\infty$.

Clearly,  $\sum_{n} r_{n} x_{n}$ converges a.e. is another way to say that $\sum_{n} \varepsilon_{n} x_{n}$ is a.s.-sign convergent.
Then taking $X=\hp$ and $x_{n} = a_{n} n^{-s} \in \hp$ we can reformulate our space $\hpr$ defined in \eqref{ostbahnhof} as follows:
\[
\hpr = \Big\{ \textstyle\sum a_{n} n^{-s} \colon \sum_{n} r_{n} a_{n} n^{-s} \in L_{1} ([0,1] ; \hp) \Big\} \,,
\]
and define the norm
\begin{equation} \label{spree}
\Big\Vert \sum a_{n} n^{-s} \Big\Vert_{\hpr} := \int_{0}^{1} \Big\Vert \sum_{n} r_{n} (t) a_{n} n^{-s} \Big\Vert_{\hp} dt \,.
\end{equation}
 We need $\hpr$ to be a Banach space. This follows from the following general result. First, recall (see \cite[page~233]{DiJaTo95}) that for a given Banach space $X$, the space  $\Rad(X)$  of almost unconditionally summable sequences $(x_n)_n $ in $X$ together with the norm
$$\Vert (x_{n})_{n} \Vert_{\Rad(X)} = \int_{0}^{1} \Big\Vert \sum_{n} r_{n} (t) x_{n} \Big\Vert_{X} dt$$ forms  a Banach space.

\begin{lemma} \label{spandauer}
Let $X$ be a Banach space and let $Y_{n}\,, n \in \mathbb{N}$ be closed subspaces of $X$. Then $Y=\{ (x_{n})_{n} \in \Rad (X) \colon x_{n} \in Y_{n} \}$ is closed in $\Rad (X)$.
\end{lemma}
\begin{proof}
Let us observe first that if $x=(x_{n})_{n} \in \Rad(X)$, then $\sum_{m} r_{m} x_{m} \in L_{1}([0,1];X)$. Due to the orthogonality of the Rademacher system we have
\[
x_{n} = \sum_{m} x_{m} \int_{0}^{1} r_{m}(t) r_{n} (t) dt =   \int_{0}^{1} \Big( \sum_{m} x_{m}r_{m}(t) \Big) r_{n} (t) dt \,,
\]
and this gives
\[
\Vert x_{n} \Vert_{X} \leq  \int_{0}^{1} \Big\Vert \sum_{m} x_{m}r_{m}(t) \Big\Vert_{X} \vert r_{n} (t) \vert dt=
 \int_{0}^{1} \Big\Vert \sum_{m} x_{m}r_{m}(t) \Big\Vert_{X} dt = \Vert x \Vert_{\Rad(X)} \,.
\]
Let us take now $(x^{(m)})_{m} \in Y$ that converges to a certain $x$ in $\Rad (X)$. We write $x^{(m)} = (x^{(m)}_{n})_{n}$ and $x=(x_{n})_{n}$, then $\Vert x^{(m)}_{n} - x_{n} \Vert_{X} \leq \Vert x^{(m)} - x \Vert_{\Rad(X)}$, and hence, for each fixed $n$,
the sequence $x^{(m)}_{n}$ converges to $x_{n}$ as $m \to \infty$. Since $x^{(m)}_{n} \in Y_{n}$ for every $n$ and all $Y_n$ are closed, we have $x_{n} \in Y_{n}$ for all $n$, or equivalently  $x \in Y$.
\end{proof}

\begin{proposition} \label{potsdamer}
For every $1 \leq p \leq \infty$ and every Banach space $E$ the space $\hpr$ endowed with the norm defined in \eqref{spree} is a Banach space.
\end{proposition}
\begin{proof}
Note that our space $\hpr$ is actually a subspace of $\Rad (\hp)$:
\[
\hpr = \big\{ (D_{n})_{n} \in \Rad (\hp) \colon  D_{n} = a_{n} n^{-s} \,,\,\, a_{n} \in E  \big\} \,.
\]
Observe that each $F_{n} = \{a_{n} n^{-s} \colon a_{n} \in E   \} \subseteq \hp$ is isometric to $E$ and hence closed. This by Lemma~\ref{spandauer} completes the proof.
\end{proof}

The following result  is crucial for the modern theory of Dirichlet series: For each $1 \leq p \leq \infty $ there is a constant $C_{p}>1$ such that  for any Banach space $E$ and
any Dirichlet series $\sum a_n n^{-s} \in \mathcal{H}_p(E)$ we for every $N$ have
\begin{equation} \label{gendarmenmarkt}
\Big\Vert \sum_{n=1}^{N} a_{n} n^{-s} \Big\Vert_{\mathcal{H}_{p} (E)}
\leq C_{p} \log N \Big\Vert \sum_{n=1}^{\infty} a_{n} n^{-s} \Big\Vert_{\mathcal{H}_{p} (E)} \,.
\end{equation}
 For $E= \mathbb{C}$ and $p = \infty$   this is a quantification of \eqref{Bohrfun} given in \cite[Lemma~1.1]{BaCaQu06} (see also \cite[Theorem~6.2.2]{QuQu13}), and for  $E= \mathbb{C}$ and $p=1$ it is \cite[Theorem~3.2]{BaQuSe15}.
 For $E= \mathbb{C}$ and $1 < p <\infty$ the situation is even better, since by \cite{AlOlSa12}, the system $(n^{-s})_{n \in \mathbb{N}}$ then forms a Schauder basis of $\mathcal{H}_{p} (\mathbb{C})$;
 hence in this situation the $\log$-term even disappears. The  vector-valued case needs an alternative approach -- see  \cite{DeGaMaSe14} for a proof which again is very much in the spirit of the starting case $E= \mathbb{C}$ and $p = \infty$ (so of Bohr's original ideas). For our new spaces $ \hpr$ the situation is much simpler.
\begin{proposition} \label{unterdenlinden}
If $1 \leq p \leq \infty$, $E$ is a Banach space and  $\sum a_{n} n^{-s} \in \hpr$, then for every $N$ we have
\[
\Big\Vert \sum_{n=1}^{N} a_{n} n^{-s} \Big\Vert_{\hpr}
\leq  \Big\Vert \sum_{n=1}^{\infty} a_{n} n^{-s} \Big\Vert_{\hpr} \,.
\]
Moreover, the sequence of partial sums converges to $\sum_{n=1}^{\infty} a_{n} n^{-s}$ in $\hpr$.
\end{proposition}
\begin{proof}
Let us fix $\sum a_{n} n^{-s} \in \hpr$ and $N \in \mathbb{N}$. We define $\lambda_{n} = 1$ for $1 \leq n \leq N$ and $\lambda_{n} = 0$ for $n >N$. We use now the Contraction Principle (see e.g. \cite[Theorem~12.2]{DiJaTo95}) to get that for $N<M$,
\begin{multline*}
\Big\Vert \sum_{n=1}^{N} a_{n} n^{-s} \Big\Vert_{\hpr}
= \int_{0}^{1} \Big\Vert \sum_{n=1}^{N} r_{n}(t) a_{n} n^{-s} \Big\Vert_{\hp} dt
= \int_{0}^{1} \Big\Vert \sum_{n=1}^{M} \lambda_{n} r_{n}(t) a_{n} n^{-s} \Big\Vert_{\hp} dt
\\
\leq \int_{0}^{1} \Big\Vert \sum_{n=1}^{M} r_{n}(t) a_{n} n^{-s} \Big\Vert_{\hp} dt
= \Big\Vert \sum_{n=1}^{M} a_{n} n^{-s} \Big\Vert_{\hpr} \,.
\end{multline*}
By \cite[Theorem~12.3]{DiJaTo95} the series $\sum r_{n} a_{n} n^{-s}$ converges in $L_{1} ([0,1]; \hp)$, hence
\[
\Big\Vert \sum_{n=1}^{M} r_{n} a_{n} n^{-s} \Big\Vert_{L_{1}}
\longrightarrow \Big\Vert \sum_{n=1}^{\infty} r_{n} a_{n} n^{-s} \Big\Vert_{L_{1}} \,\,\, \text{ as }\,\,\,M \to \infty\,.
\]
By the very definition of the norm in $\hpr$ this gives the conclusion.
\end{proof}

Our next result shows that for $1 \leq p < \infty$  the  study of $\mathcal{H}_{p}^{\mathrm{rad}} (\mathbb{C})$ reduces to the study of $\mathcal{H}_2$ (see also Proposition~\ref{europaplatz} for a vector-valued extension).
\begin{proposition} \label{brandenburgertor}
For  $1 \leq p < \infty$ we have $\mathcal{H}_{p}^{\mathrm{rad}} (\mathbb{C}) = \mathcal{H}_2$ \,.
\end{proposition}
\begin{proof}
For fixed  $N \in \mathbb{N}$  use the definition of $\mathcal{H}_{p}^{\mathrm{rad}} (\mathbb{C})$,  Kahane's inequality (see e.g. \cite[Theorem~11.1]{DiJaTo95}), the definition of $\mathcal{H}_{p} (\mathbb{C})$, and finally Khinchin's inequality (see e.g.\cite[Theorem~1.10]{DiJaTo95}) in order to get
\begin{align*}
\Big\Vert \sum_{n=1}^{N} a_{n} n^{-s} \Big\Vert_{\mathcal{H}_{p}^{\mathrm{rad}} (\mathbb{C})}
& = \int_{0}^{1} \Big\Vert \sum_{n=1}^{N} r_{n} (t) a_{n} n^{-s} \Big\Vert_{\mathcal{H}_{p} (\mathbb{C})} dt
\sim \bigg(  \int_{0}^{1} \Big\Vert \sum_{n=1}^{N} r_{n} (t) a_{n} n^{-s} \Big\Vert_{\mathcal{H}_{p} (\mathbb{C})}^{p} dt \bigg)^{\frac{1}{p}} \\
&= \bigg( \int_{0}^{1} \int_{\mathbb{T}^{\mathbb{N}}} \Big\vert \sum_{\alpha} r_{p^{\alpha}} (t) a_{p^{\alpha}} z^{\alpha} \Big\vert^{p} dz dt \bigg)^{\frac{1}{p}}
= \bigg( \int_{\mathbb{T}^{\mathbb{N}}} \int_{0}^{1} \Big\vert \sum_{\alpha} r_{p^{\alpha}} (t) a_{p^{\alpha}} z^{\alpha} \Big\vert^{p} dt dz  \bigg)^{\frac{1}{p}}\\
& \sim \bigg( \int_{\mathbb{T}^{\mathbb{N}}} \Big( \sum_{\alpha} \vert a_{p^{\alpha}} z^{\alpha} \vert^{2} \Big)^{\frac{p}{2}} dz  \bigg)^{\frac{1}{p}}
= \bigg( \int_{\mathbb{T}^{\mathbb{N}}} \Big( \sum_{\alpha} \vert a_{p^{\alpha}} \vert^{2} \Big)^{\frac{p}{2}} dz  \bigg)^{\frac{1}{p}} \\
& = \Big( \sum_{\alpha} \vert a_{p^{\alpha}} \vert^{2} \Big)^{\frac{1}{2}}
=  \Big( \sum_{n=1}^{N}  \vert a_{n} \vert^{2} \Big)^{\frac{1}{2}}.
\end{align*}
This gives the conclusion.
\end{proof}

It is not surprising  that in the vector-valued situation such a   description of $\mathcal{H}_{p}^{\mathrm{rad}} (E)$ is more involved. However, if the space $E$ is nice enough we do have something. Let us recall
\cite[page~46]{LiTz79} that a Banach lattice $E$ is $q$-concave (with $1 \leq q < \infty$) if there exists a constant $C>0$ such that for every choice $x_{1}, \ldots , x_{N} \in E$
\[
\bigg( \sum_{n=1}^{N} \Vert x_{n} \Vert^{q}  \bigg)^{\frac{1}{q}} \leq C \bigg\Vert \Big(   \sum_{n=1}^{N} \vert x_{n} \vert^{q} \Big)^{\frac{1}{q}} \bigg\Vert \,.
\]
For a Banach lattice $E$ we define $\widetilde{E(\ell_{2})}$ to be the space of sequences $(x_{n})_{n=1}^{\infty}$ in $E$ such that
\[
\Vert (x_{n})_{n} \Vert_{\widetilde{E(\ell_{2})}}  = \sup_{N} \bigg\Vert \Big(   \sum_{n=1}^{N} \vert x_{n} \vert^{2} \Big)^{\frac{1}{2}} \bigg\Vert_{E} < \infty \,.
\]
The closure in $\widetilde{E(\ell_{2})}$ of the subspace of finite sequences is denoted by $E(\ell_2)$. We remark that these two spaces coincide if and only if $E$ is weakly sequentially complete (see \cite[p. 46]{LiTz79} for details).
\begin{proposition} \label{europaplatz}
If $E$ is a Banach lattice that is $q$ concave for some $q$, then $\hpr = E(\ell_{2})$ for every $1 \leq p < \infty$.
\end{proposition}
\begin{proof}
Let us first consider $a_{1}, \ldots , a_{N} \in E$. By the very definition of the norms in $\hpr$ and $\hp$ and
Kahane's inequality (that we apply twice) we have (with constants independent of $N$)
\begin{multline*}
\Big\Vert \sum_{n=1}^{N} a_{n} n^{-s} \Big\Vert_{\hpr}
= \int_{0}^{1} \Big\Vert \sum_{n=1}^{N} r_{n} (t) a_{n} n^{-s} \Big\Vert_{\hp} dt
\sim  \bigg( \int_{0}^{1} \Big\Vert \sum_{n=1}^{N} r_{n} (t) a_{n} n^{-s} \Big\Vert^{p}_{\hp} dt\bigg)^{\frac{1}{p}} \\
= \bigg( \int_{0}^{1} \int_{\mathbb{T}^{N}} \Big\Vert \sum_{\alpha} r_{p^{\alpha}} (t) a_{p^{\alpha}} z^{\alpha} \Big\Vert_{E}^{p} dz   dt \bigg)^{\frac{1}{p}}
=  \bigg( \int_{\mathbb{T}^{N}} \int_{0}^{1} \Big\Vert \sum_{\alpha} r_{p^{\alpha}} (t) a_{p^{\alpha}} z^{\alpha} \Big\Vert_{E}^{p}    dt dz\bigg)^{\frac{1}{p}} \\
\sim \bigg(\int_{\mathbb{T}^{N}} \Big( \int_{0}^{1}  \Big\Vert \sum_{\alpha} r_{p^{\alpha}} (t) a_{p^{\alpha}} z^{\alpha} \Big\Vert_{E}  dt\Big)^{p} dz \bigg)^{\frac{1}{p}} \,.
\end{multline*}
But now, since $E$ is $q$ concave for some $q$, for each fixed $z \in \mathbb{T}^{N}$ we have by \cite[Theorem~1.d.6]{LiTz79}
\[
\int_{0}^{1}  \Big\Vert \sum_{\alpha} r_{p^{\alpha}} (t) a_{p^{\alpha}} z^{\alpha} \Big\Vert_{E}  dt
\sim \Big\Vert \Big( \sum_{\alpha} |a_{p^{\alpha}} z^{\alpha}|^2 \Big)^{\frac{1}{2}} \Big\Vert_{E}
= \Big\Vert \Big( \sum_{n=1}^{N}  |a_{n} |^2 \Big)^{\frac{1}{2}} \Big\Vert_{E} \,.
\]
This, together with Proposition~\ref{unterdenlinden}, yields the conclusion.
\end{proof}

\section{A reformulation of $\boldsymbol{\spr}$}

\noindent Maurizi and Queff\'elec showed in \cite[Theorem~2.4]{MaQu10} how $S$ can be characterized in terms of bounds of the norm of the partial sums. A  modification of their argument using \cite{DeGaMaSe14} gives the following vector-valued version: For every $1\leq p \leq \infty$ and every Banach space $E$
\begin{equation} \label{leipziger}
S_{p}(E) = \inf \Big\{ \sigma >0 \, \big|\,\exists c_{\sigma}\,
\forall
D=  \sum_{n=1}^{N} a_{n} n^{-s} \in \hp
: \,\sum_{n=1}^{N} \Vert a_{n} \Vert \leq c_{\sigma} N^{\sigma} \Vert D \Vert_{\hp}  \Big\}\,.
\end{equation}
 The original proof of \cite[Theorem~2.4]{MaQu10} for $E = \mathbb{C}$ uses two key tools. The proof of one inequality is based on a closed-graph argument using the fact that $\mathcal{H}_p$ is Banach, and the proof of the converse inequality  relies on  \eqref{gendarmenmarkt}.
The results from the preceding section prepare  us well to establish the following  analogue of \eqref{leipziger} within our  setting.
\begin{proposition} \label{oranienburg} For every $1\leq p \leq \infty$ and Banach space $E$ we have
\begin{equation} \label{tegel}
\spr = \inf \Big\{ \sigma >0\,\, \big|\,\, \exists c_{\sigma} \,
\forall
D=  \sum_{n=1}^{N} a_{n} n^{-s} \in \hpr
: \,\sum_{n=1}^{N} \Vert a_{n} \Vert \leq c_{\sigma} N^{\sigma} \Vert D \Vert_{\hpr}  \Big\} \,.
\end{equation}
\end{proposition}

\begin{proof}
To show one inequality, let us take $\sigma > \spr$. A closed-graph argument (here we need Proposition~\ref{potsdamer}) gives that there exists $c_{\sigma}>0$ such that
\[
\sum_{n=1}^{\infty} \Vert a_{n} \Vert \frac{1}{n^{\sigma}} \leq c_{\sigma} \big\Vert \textstyle \sum a_{n} n^{-s} \big\Vert_{\hpr}
\]
for every $\sum a_{n} n^{-s} \in \hpr$. Then, for given $a_{1}, \ldots a_{N} \in E$ we have
\[
\sum_{n=1}^{N} \Vert a_{n} \Vert \leq N^{\sigma}  \sum_{n=1}^{N} \frac{\Vert a_{n} \Vert}{n^{\sigma}}\le
c_{\sigma} N^{\sigma}  \big\Vert \textstyle \sum_{n=1}^{N} a_{n} n^{-s} \big\Vert_{\hpr} \,.
\]
Let us conversely fix now some $\sigma_{0}>0$ satisfying the inequality in Proposition~\ref{oranienburg}, and choose $\sum a_{n} n^{-s} \in \hpr$. By Abel's summation and Proposition~\ref{unterdenlinden} we have, for any $\sigma>\sigma_0$,
\begin{multline*}
\sum_{n=1}^{N} \Vert a_{n} \Vert \frac{1}{n^{\sigma}}
= \sum_{n=1}^{N} \Vert a_{n} \Vert \frac{1}{N^{\sigma}} + \sum_{n=1}^{N-1} \sum_{k=1}^{n} \Vert a_{k} \Vert \Big( \frac{1}{n^{\sigma}} - \frac{1}{(n+1)^{\sigma}} \Big) \\
\leq c_{\sigma_{0}} N^{\sigma_{0} - \sigma} \Vert \sum  a_{n} n^{-s} \Vert_{\hpr}
+ \sum_{n=1}^{N-1} c_{\sigma_{0}} n^{\sigma_{0}} \Vert \sum  a_{n} n^{-s} \Vert_{\hpr}\Big( \frac{1}{n^{\sigma}} - \frac{1}{(n+1)^{\sigma}} \Big) \,.
\end{multline*}
Standard computations following  \cite[Lemma~1.1]{BaCaQu06} finally give
\[
\sum_{n=1}^{N} \Vert a_{n} \Vert \frac{1}{n^{\sigma}}
\leq c_{\sigma_{0}} \Vert \sum  a_{n} n^{-s} \Vert_{\hpr} \Big( 1 + \sum_{n=1}^{\infty} \frac{\sigma}{n^{\sigma-\sigma_{0}+1}} \Big)\,.
\]
Hence $\spr \leq \sigma$ and, since $ \sigma$ was arbitrary, the proof is completed.
\end{proof}

As an immediate consequence we obtain the following

\begin{corollary}\label{friedrichstr}
For every $1\leq p \leq \infty$ and Banach space $E$ we have $\spr \leq S_{p}(E)$\,.
\end{corollary}
\begin{proof}
Let us take $\sigma >0$ satisfying the condition in \eqref{leipziger}. Then for every choice of finitely many $a_1, \ldots, a_N \in E$ and every $t \in [0,1]$ we have
\[
\sum_{n=1}^{N} \Vert a_{n} \Vert = \sum_{n=1}^{N} \Vert r_{n}(t) a_{n} \Vert
\leq c_{\sigma} N^{\sigma} \Big\Vert \sum_{n=1}^{N} r_{n}(t) a_{n} n^{-s} \Big\Vert_{\hp}  \,.
\]
Integration with respect to $t$ and Proposition~\ref{oranienburg} give the conclusion.
\end{proof}

\section{Uniform a.s.-sign convergence versus absolut covergence for  Hardy-type Dirichlet series}

The following theorem is our first main result.

\begin{theorem}\label{alexanderplatz1}
For every Banach space $E$ and $1 \leq p \leq \infty$ we have
\[
\spr = 1 - \frac{1}{\cot E} \,,
\]
i.e., if a Dirichlet series $\sum a_n n^{-s} \in \mathfrak{D}(E)$ is a.e-sign convergent in $\mathcal{H}_p(E)$, then
$\sum_n \|a_n\|_E \,n^{-\sigma} < \infty$ for $\sigma > \sigma_0 :=  1 - \frac{1}{\cot E}$, and  $\sigma_0 $ is best possible.
\end{theorem}

\vspace{3mm}

\noindent Note that this,  in combination with \eqref{pariserplatz}, in particular shows that for each Banach space $E$ and each $1 \leq p \leq \infty$  we have
 \[
 \spr =S_p(E).
 \]
In view of Corollary~\ref{friedrichstr} and \eqref{pariserplatz},
we only have to take care of the lower estimate.

\bigskip
In order to prove  Theorem~\ref{alexanderplatz1}, we need the concept of $m$-homogeneous Dirichlet series (that was first suggested in \cite{BoHi31} and whose set we denote by $\mathfrak{D}_{m}(E)$): Those $\sum a_{n} n^{-s}$ for which $a_{n} \neq 0$ only if $n$ has exactly $m$ prime divisors (counted with multiplicity, we denote this by $\Omega(n)=m$). Then
we define $$\mathcal{H}_{p,m}(E) \,\,\,\text{ and }\,\,\,\hprm\,,$$
 to be the (closed) subspace of  $\mathcal{H}_{p}(E)$ and $\hpr$, respectively,  consisting of $m$-homogeneous Dirichlet series.
 It is well-known that for all $1 \leq p,q < \infty$  and $m$ (see e.g. \cite[Theorem~9.1]{CoGa86} or~ \cite{Ba02})
\begin{align} \label{pm=qm}
 \mathcal{H}_{p,m}(\mathbb{C})=\mathcal{H}_{q,m}(\mathbb{C})\,.
 \end{align}
 We now can repeat the above program and define for every $m \in \mathbb{N}$,
every $1 \leq p \leq \infty$ and every Banach space $E$
\begin{gather*}
S_{p,m} (E)
:=  \sup_{D \in \mathfrak{D}(E)\text{ $m$-hom.}} \sigma_{a} (D) - \sigma_{\mathcal{H}_p} (D) \\
S_{p,m}^{\mathrm{rad}} (E)
:= \sup_{D \in \mathfrak{D}(E)\text{ $m$-hom.}}  \sigma_{a} (D) - \srh (D) ;
\end{gather*}
 obviously $S_{p,m} (E) \leq S_p(E)$ and $S_{p,m}^{\mathrm{rad}} (E) \leq S_{p}^{\mathrm{rad}} (E)$. Exactly as above (see the proof of Proposition~\ref{oranienburg}), we may show that
\begin{align} \label{tegel1}
S_{p,m} (E) = \inf \Big\{ \sigma >0\,\, \big|\,\, \exists c_{\sigma} \,
\forall
D=  \sum_{n=1}^{N} a_{n} n^{-s} \in \mathcal{H}_{p,m}(E): \,  \sum_{n=1}^{N} \|a_{n}\| \leq  c_{\sigma} N^{\sigma} \|D \|_{\mathcal{H}_{p,m}(E)}  \Big\} \,
\end{align}
and
\begin{align} \label{tegel2}
S_{p,m}^{\mathrm{rad}} (E) = \inf \Big\{ \sigma >0\,\, \big|\,\, \exists c_{\sigma} \,
\forall
D=  \sum_{n=1}^{N} a_{n} n^{-s} \in \hprm: \,  \sum_{n=1}^{N} \|a_{n}\| \leq c_{\sigma} N^{\sigma} \|D \|_{\hprm}  \Big\}\,.
\end{align}
Moreover, following the argument for Corollary~\ref{friedrichstr} we have
\begin{equation} \label{kanzleramt}
S_{p,m}^{\mathrm{rad}} (E) \leq S_{p,m} (E)\,.
\end{equation}

 As a by product of our proof (see the end of Subsection~\ref{end}) we are going to obtain the following result (for the analogue  for finite dimensional spaces
 see Proposition ~\ref{17juni}):
  \begin{proposition} \label{contrast} For every infinite dimensional Banach space $E$ and every $m$
\[
S_{p,m}^{\mathrm{rad}} (E)  =S_{p,m}(E) = 1 - \frac{1}{\cot E} \,.
\]
\end{proposition}

\smallskip
We divide the proof of Theorem~\ref{alexanderplatz1} into two separate cases: for finite and infinite dimensional spaces.

\subsection{The finite dimensional case} \label{end}
For every finite dimensional Banach space $E$ we have $\cot E = 2$. Then the following counterpart of \eqref{contrast}   obviously implies the lower bound in Theorem~\ref{alexanderplatz1}.

\begin{proposition} \label{17juni} For every finite dimensional Banach space $E$ and every $m$
\[
S_{p,m}^{\mathrm{rad}} (E)  = S_{p,m} (E) =
\begin{cases}
\frac{1}{2} & \text{ for } 1 \leq p < \infty \\
\frac{m-1}{2m} & \text{ for } p = \infty\,.
\end{cases}
\]
\end{proposition}
\noindent For $p=\infty$ and  $E= \mathbb{C}$ this result is due to Bohnenblust-Hille \cite{BoHi31} and Hartman \cite{Ha39}.
\begin{proof}
Since $S_{p,m}^{rad} (E)$ is invariant under renorming of $E$, we may assume that $E= \ell_2^k$ (i.e. $\mathbb{C}^{k}$ with the euclidean norm).
By \eqref{kanzleramt} we need to  show the proper lower bound for $S_{p,m}^{\mathrm{rad}} (\ell_2^k)$ and the  proper  upper bound for $S_{p,m} (\ell_2^k)$. We start with the upper bound for  $S_{p,m} (\ell_2^k)$:
Assume first that $1 \leq p < \infty$.  Given $a_1, \ldots, a_N \in \ell_2^k$, we then conclude from
the Cauchy-Schwarz inequality and Kahane's inequality that
\begin{multline*}
\sum_{k=1}^N \|a_n\|_{\ell_2^k}
 \leq N^{1/2} \left(\sum_{k=1}^N \|a_n\|_{\ell_2^k}^2\right)^{1/2}
= N^{1/2} \left( \int_{\mathbb{T}^\mathbb{N}}  \Big\|   \sum_{|\alpha|=m} a_{p^\alpha} z_\alpha \Big\|^2_{\ell_2^k}\right)^{1/2}
\\
\sim
N^{1/2} \left( \int_{\mathbb{T}^\mathbb{N}}  \Big\|   \sum_{|\alpha|=m} a_{p^\alpha} z_\alpha \Big\|^p_{\ell_2^k}\right)^{1/p}
=
  N^{1/2} \left\| \sum_{n=1}^N a_n n^{-s}\right\|_{\mathcal{H}_{p,m}(\ell_2^k)}\,,
\end{multline*}
which by \eqref{tegel1} shows what we want. Now for $p= \infty$ we conclude from H\"older's inequality and the polynomial Bohnenblust-Hille inequality (in the form of \cite[Theorem~5.3]{DeMaSc12}) that
\begin{multline*}
\sum_{k=1}^N \|a_n\|_{\ell_2^k}
 \leq N^{\frac{2m}{m-1}} \left(\sum_{k=1}^N \|a_n\|_{\ell_2^k}^{\frac{2m}{m+1}}\right)^{\frac{m+1}{2m}}
\\
\leq     C^m  N^{\frac{2m}{m-1}}\sup_{z \in \mathbb{T}^\mathbb{N}} \Big\|   \sum_{|\alpha|=m} a_{p^\alpha} z_\alpha \Big\|_{\ell_2^k}
=  C^m  N^{\frac{2m}{m-1}} \Big\| \sum_{n=1}^N a_n n^{-s}\Big\|_{\mathcal{H}_{\infty,m}(\ell_2^k)}\,,
\end{multline*}
and again  \eqref{tegel1} gives the conclusion.\\
Let us turn to the lower bound of $S_{p,m}^{\mathrm{rad}} (\ell_2^k) $: A simple argument shows that $ S_{p,m}^{\mathrm{rad}} (\mathbb{C})\leq S_{p,m}^{\mathrm{rad}} (\ell_2^k) $\,,
so it remains to estimate  $ S_{p,m}^{\mathrm{rad}}(\mathbb{C})$ from below. We again start with the case $1 \leq p < \infty$. Then we know from \eqref{pm=qm}
that $ S_{p,m}^{\mathrm{rad}}(\mathbb{C})= S_{2,m}^{\mathrm{rad}}(\mathbb{C})$, and hence we may concentrate on the case $p=2$.

Clearly $S_{2,m}^{\mathrm{rad}}(\mathbb{C}) \geq
S_{2,1}^{\mathrm{rad}}(\mathbb{C})$, then we can assume that $\sigma >0$ and $c_\sigma >0$ are as in \eqref{tegel1} with $p=2$, $m=1$  and $E= \mathbb{C}$. Hence  by the prime number theorem there are constants $C_1, C_2 >0$
such that
\begin{equation} \label{pergamon}
\frac{N}{\log N} \leq C_1  \sum_{\substack{n=1 \\ \Omega(n)=1}}^N 1 \leq c_\sigma N^{\sigma}  \Big\| \sum_{\substack{n=1 \\ \Omega(n)=1}}^N  n^{-s}\Big\|_{\mathcal{H}_{2,m}^{\mathrm{rad}}(\mathbb{C})} \leq C_2 c_\sigma N^{\sigma} \Big( \frac{N}{\log N} \Big)^{1/2}\,,
\end{equation}
and this is exactly what we need.

Finally, we consider the case $p= \infty$:  We fix $\sigma >S_{\infty,m}^{\mathrm{rad}}(\mathbb{C}) $; by a standard closed graph argument there is a constant $c_{\sigma}>0$ such that
\begin{equation} \label{nikolai}
\sum_{n=1}^{\infty} \vert a_{n} \vert \frac{1}{n^{\sigma}} \leq c_{\sigma} \big\Vert \textstyle \sum_n a_{n} n^{-s} \big\Vert_{\mathcal{H}_{\infty}^{\mathrm{rad}} (\mathbb{C})} \,.
\end{equation}
We consider $\varepsilon_{\alpha}$ independent Rademacher random variables (i.e. each one taking values $\pm 1$ with probability $1/2$) for $\alpha \in \mathbb{N}_{0}^{N}$ with $\vert \alpha \vert =m$. By the Kahane-Salem-Zygmund inequality, as presented in \cite[Theorem~5.3.4]{QuQu13} there is a constant $C>0$ such that
\[
\int \sup_{z \in \mathbb{D}^{N}} \Big\vert \sum_{\substack{\alpha \in \mathbb{N}_{0}^{N} \\ \vert \alpha \vert =m}}  \varepsilon_{\alpha}(\omega) z^{\alpha} \Big\vert d \omega
 \leq C \Big( \sum_{\substack{\alpha \in \mathbb{N}_{0}^{N} \\ \vert \alpha \vert =m}} 1 \Big)^{\frac{1}{2}} \sqrt{N \log m}
\leq C N^{\frac{m+1}{2}} \sqrt{\log m} \,.
\]
We consider now the polynomial $\sum_{\substack{\alpha \in \mathbb{N}_{0}^{N} \\ \vert \alpha \vert =m}} z^{\alpha}$ and denote by $D_{N}$ the Dirichlet series associated to it by \eqref{willybrandt}. Then
\[
\Vert D_{N} \Vert_{\mathcal{H}_{\infty}^{\mathrm{rad}} (\mathbb{C})}
= \int_{0}^{1} \sup_{z \in \mathbb{D}^{N}} \Big\Vert \sum_{\substack{\alpha \in \mathbb{N}_{0}^{N} \\ \vert \alpha \vert =m}} r_{p^{\alpha}} (t) z^{\alpha} \big\Vert dt \leq C N^{\frac{m+1}{2}} \sqrt{\log m} \,.
\]
With this and \eqref{nikolai} we get that, for every $N$
\[
\sum_{\substack{\alpha \in \mathbb{N}_{0}^{N} \\ \vert \alpha \vert =m}} \frac{1}{p^{\alpha \sigma}}
\leq c_{\sigma} C N^{\frac{m+1}{2}} \sqrt{\log m} \,.
\]
All we need now is a lower bound of $\sum_{\vert \alpha \vert=m} \frac{1}{p^{\alpha \sigma}}$. By  a weak consequence of the Prime Number Theorem~$p_{j} \sim j \log j$. Then for a fixed $\varepsilon >0$ there is a constant $B>0$
such that for all $j$ we have $p_{j} \leq B j^{1+ \varepsilon}$, hence
\[
\sum_{\vert \alpha \vert=m} \frac{1}{p^{\alpha \sigma}}
= \sum_{1 \leq j_{1} \leq \ldots \leq j_{m} \leq N} \frac{1}{(p_{j_{1}} \cdots p_{j_{m}})^{\sigma}}
\geq \frac{1}{B^{m}}\sum_{1 \leq j_{1} \leq \ldots \leq j_{m} \leq N} \frac{1}{(j_{1}\ldots j_{N} )^{(1+\varepsilon)\sigma}} \,.
\]
Let us now observe that
\[
\sum_{j_{1},\ldots , j_{m} =1}^{N} \frac{1}{(j_{1}\cdots j_{m} )^{(1+\varepsilon)\sigma}}
\leq \sum_{1 \leq j_{1} \leq \ldots \leq j_{m} \leq N} m! \frac{1}{(j_{1}\cdots j_{m} )^{(1+\varepsilon)\sigma}} \, .
\]
Then
\begin{align*}
 \sum_{1 \leq j_{1} \leq \ldots \leq j_{m} \leq N} \frac{1}{(j_{1}\ldots j_{m})^{(1+\varepsilon)\sigma}}
 &
\geq \frac{1}{m!}  \sum_{j_{1}, \ldots, j_{N}=1}^{N} \frac{1}{(j_{1}\ldots j_{m} )^{(1+\varepsilon)\sigma}}
\\
&
=  \frac{1}{m!} \bigg(  \sum_{j=1}^{N} \frac{1}{j^{(1+\varepsilon)\sigma}} \bigg)^{m}
\geq D \frac{N^{m}}{N^{(1+\varepsilon)\sigma m}} \, .
\end{align*}
This altogether gives that there is a constant $K$ depending only on $m$ such that
\[
N^{m(1-(1+ \varepsilon)\sigma)} \leq K \sqrt{\log m} N^{\frac{m+1}{2}}\,,
\]
which yields $\frac{m-1}{2m} \leq \sigma$ and  gives the result.
\end{proof}

\vspace{2mm}

\subsection{The infinite dimensional case}

\noindent Let us now prove Theorem~\ref{alexanderplatz1} for
infinite dimensional Banach spaces $E$. Once again, by Corollary~\ref{friedrichstr} and equation \eqref{pariserplatz},
it suffices to check the following:
Given an infinite dimensional Banach space $E$ and  $1 \leq p \leq \infty$ the following holds
\begin{equation}\label{bahnhof zoo}
1 - \frac{1}{\cot E}\leq \spr  \,.
\end{equation}

\begin{proof}
For each fixed $t \in [0,1]$ we have
\[
\Big\Vert \sum r_{n}(t) a_{n} n^{-s} \Big\Vert_{\hp}
\leq \Big\Vert \sum r_{n}(t) a_{n} n^{-s} \Big\Vert_{\mathcal{H}_{\infty}(E)} \,.
\]
Integrating with respect to $t$ we get that  $\mathcal{H}_{\infty}^{\mathrm{rad}}(E) \subset \hpr$ for every $1\leq p \leq \infty$. Hence to find a lower bound for $\spr$ it is enough to get some lower estimate for $S_{\infty}^{\mathrm{rad}}(E)$.
What we are going to do is to work only with 1-homogeneous Dirichlet series, finding lower bounds for $S_{\infty,1}^{\mathrm{rad}} (E)$. Recall from \eqref{tegel1} that
\begin{equation*}\label{radp1}
S_{\infty,1}^{\mathrm{rad}} (E) = \inf \Big\{ \sigma >0 \colon \exists c_{\sigma} \, \, \forall
\,a_{p_{1}}, \dots, a_{p_{N}} \in E\,:\, \,
\sum_{k=1}^{N} \Vert a_{p_{k}} \Vert \leq c_{\sigma} p_{N}^{\sigma} \big\Vert \sum_{k=1}^{N} a_{p_{k}} p_{k}^{-s} \big\Vert_{\mathcal{H}_{\infty}^{\mathrm{rad}}(E)} \Big\} \,.
\end{equation*}
On the other hand for each $t$,
$$
\Big\Vert \sum_{k=1}^{N} r_{p_{k}}(t) a_{p_{k}} p_{k}^{-s} \Big\Vert_{\mathcal{H}_{\infty}(E)}  = \sup_{u \in \mathbb{T}^{N}} \Big\Vert \sum_{k=1}^{N} r_{p_{k}}(t) a_{p_{k}} u_{k} \Big\Vert_{E}
 =  \sup_{w \in \mathbb{T}^{N}} \Big\Vert \sum_{k=1}^{N} a_{p_{k}} w_{k} \Big\Vert_{E} = \Big\Vert \sum_{k=1}^{N}   a_{p_{k}} p_{k}^{-s} \Big\Vert_{\mathcal{H}_{\infty}(E)} \,.
$$
Now, integrating on $t$ we obtain
\begin{align*}\Big\Vert \sum_{k=1}^{N}   a_{p_{k}} p_{k}^{-s} \Big\Vert_{\mathcal{H}_{\infty}^{\mathrm{rad}}(E)}
= \int_{0}^{1} \Big\Vert \sum_{k=1}^{N} r_{p_{k}}(t) a_{p_{k}} p_{k}^{-s} \Big\Vert_{\mathcal{H}_{\infty}(E)} dt
 =  \Big\Vert \sum_{k=1}^{N}   a_{p_{k}} p_{k}^{-s} \Big\Vert_{\mathcal{H}_{\infty}(E)} \,.
\end{align*}
This means that $S_{\infty,1}^{\mathrm{rad}} (E)=S_{\infty,1} (E). $
But from \cite[p.554]{DeGaMaPG08} we know that
$
S_{\infty,1} (E) = 1 - \frac{1}{\cot E}
$
which completes the proof.
\end{proof}

A brief analysis of the preceding proof shows that we also get Proposition~\ref{contrast} as a by-product.

\section{Sharp estimates}
By definition the $x$th Sidon constant for Dirichlet series is given by
\begin{equation} \label{monsterdef}
S_\infty(x) \,: =\, \sup_{(a_{n})_{n \in \mathbb{N}} \subseteq \mathbb{C}}  \frac{\sum_{n \leq x}  |a_{n}|}{  \left\|  \sum_{n \leq x}  a_n n^{-s} \right\|_{\mathcal{H}_\infty(\mathbb{C})}}
\,,
\end{equation}
and its (almost) precise asymptotic is expressed in the following formula:
\begin{equation} \label{monster}
S_\infty(x) \, =\,
 \frac{\sqrt{x}}{e^{\left( \frac{1}{\sqrt{2}} + o(1) \right)   \sqrt{\log x \log\log x}}}\,;
\end{equation}
this results with weaker constants instead of $\frac{1}{\sqrt{2}}$ was proved in \cite[Theorem~4.3]{KoQu01}, the lower estimate was given in \cite[Th\'eor\`eme~1.1]{Br08}, and finally the upper estimate followed from the hypercontractivity of the Bohnenblust-Hille inequality in  \cite[Theorem~1]{DeFrOrOuSe11}.
In view of the characterization  \eqref{leipziger}, equation \eqref{monster}   represents  a sharp estimate of  the largest possible width on which a  Dirichlet series $D=\sum a_n n^{-s}$ converges uniformly but not absolutely.
Given $x \geq 2$ and $1 \leq p \leq \infty$,  asymptotically correct estimates for
\begin{equation} \label{monster-p}
S_p(x) \,: =\, \sup_{(a_{n})_{n \in \mathbb{N}} \subseteq \mathbb{C}}   \frac{\sum_{n \leq x}  |a_k|}{  \left\|  \sum_{n \leq x}  a_n \frac{1}{n^s} \right\|_{\mathcal{H}_p(\mathbb{C})}}
\end{equation}
 like \eqref{monster} are unfortunately so far unknown for $p\ne 2$. For $p=2$ we  have $S_2(x)=\sqrt{x}$ by \eqref{HeLiSe}.
An analogue of this definition in our probabilistic setting \`a la Hartman is (again $x \in \mathbb{N}$ and $1 \leq p \leq \infty$)
\begin{equation*} \label{monster-rad-p}
S_p^{\mathrm{rad}}(x) : = \sup_{(a_{n})_{n \in \mathbb{N}} \subseteq \mathbb{C}}   \frac{\sum_{n \leq x}  |a_{n}|}{  \left\|  \sum_{n \leq x}  a_n \frac{1}{n^s} \right\|_{\mathcal{H}_p^{\mathrm{rad}}(\mathbb{C})}}  \,\,.
\end{equation*}
Proposition~\ref{oranienburg} and Theorem~\ref{alexanderplatz1} (for $E=\mathbb{C}$)  suggest the following analogue of \eqref{monster}. It can be seen  as the definite result of Hartman's original  question.

\begin{theorem}\label{main}
We have, as  $x$ tends to  $\infty$
\begin{equation*}
S_p^{\mathrm{rad}}(x)   =
\begin{cases}
O(\sqrt{x}) & \,\,\,\text{ for } \,\,\,1 \leq p < \infty \\
\frac{\sqrt{x}}{e^{\left( \frac{1}{\sqrt{2}} + o(1) \right)   \sqrt{\log x \log\log x}}} & \,\,\,\text{ for } \,\,\, p = \infty\,.
\end{cases}
\end{equation*}
\end{theorem}

\noindent The formula for $1 \leq p < \infty$ is an immediate consequence of Proposition
\ref{brandenburgertor}. Let us deal with the case $p=\infty$. We first prove that for every  $x$
and $p$
\begin{equation} \label{up}
S_p^{\mathrm{rad}}(x) \leq S_p(x);
\end{equation}
then the upper estimate for $S_\infty^{\mathrm{rad}}(x)$ obviously follows from \eqref{monster}.
  By definition $S_p^{\mathrm{rad}}(x)$ is the best constant $C >0$ such that for all sequences $(a_{n})_{n \in \mathbb{N}} \subseteq \mathbb{C}$  we have
$
\sum_{n \leq x} |a_n|\leq C\big\|   \sum_{n \leq x} a_n n^{-s}\big\|_{\mathcal{H}_p(\mathbb{C})}\,.
$
 But for each $t \in [0,1] $
\[
\sum_{n \leq x} |a_n| = \sum_{n \leq x} |a_nr_n(t)|\leq C\bigg\|   \sum_{n \leq x} r_n(t) a_n n^{-s}\bigg\|_{\mathcal{H}_p(\mathbb{C})}\,,
\]
so that  \eqref{up} follows by integration.
It remains to prove the lower estimate for $S_\infty^{\mathrm{rad}}(x) $ in Theorem~\ref{main}, and the arguments we give    follow from an analysis of the proof for \eqref{monster}. Our presentation is close to that of \cite{Br08} and  also \cite[Theorem~5.4.3]{QuQu13}, and it is mainly given for the sake of completeness.
Before we start we  need some preparation from analytic number theory.

Given $k \in \mathbb{N}$, define $\mathcal J(k) =\big\{\bj = (j_1, \dotsc, j_k) \in \mathbb{N}_0^k\,: \, 1 \leq j_1 \leq \dots \leq j_k \big\}$, and for any sequence $z=(z_n)$ of complex numbers
and any $\bj \in \mathcal J(k) $ let $z_\bj = z_{j_1} \cdots  z_{j_k}$.
Moreover, for  $x > 2$ and $2 < y \le x$, choose $\ell \in \mathbb{N}$ such that $p_\ell \le y < p_{\ell+1}$  (note that with the usual notation from number theory $\ell = \pi(y)$).
With $x$ and $\ell$ define the index set
\begin{align*}
    J^-(x;y) &= \big\{ \bj = (j_1, \dotsc, j_k) \in\mathcal J(k) \,\,:\,\, k \in \mathbb{N}, \,p_\bj \le x, \,j_k \le \ell \big\}.
\end{align*}
Note first that
  $2^{\text{length}(\bj)} \leq p_\bj \leq x\,$ for every $\bj \in  J^-(x;y ) \,,$
hence the maximal length
\begin{equation} \label{length}
L:= \max \big\{ \text{length}(\bj)\,:\,  \bj \in  J^-(x;y ) \big\} \leq \frac{\log x}{\log 2}\,.
\end{equation}
The asymptotic behavior of the function $\big| J^-(x;y )\big|$ is very well described by the so called
Dickmann function $\varrho:[0,\infty[\rightarrow \mathbb{R}$ which is uniquely determined through  the following
conditions:
\begin{itemize}
\item
$\varrho$ is differentiable on $]1,\infty[$ where it satisfies
the  differential equation
 $$u\varrho'(u)+\varrho(u-1) =0\,. $$
\item
$\varrho(u)=1$ for all $0\leq u \leq 1$, and $\varrho$ is continuous at 1.
\end{itemize}

 For this definition see e.g.  \cite[III.5, p.~365,370]{Te95}. We  need the following two asymptotic estimates; the first one can be found in   \cite[III.5.5, Corollary~9.3]{Te95} (see also \cite[Eq. (1.8)]{HiTe93}), and the second in \cite[Eq. (1.7)]{HiTe93}:
\begin{itemize}
\item
Given $\varepsilon >0$, there is $C=C(\varepsilon) >0$  such for all $x,y$ with
$x > 2$ and   $e^{(\log\log x)^{\frac{5}{3}+\varepsilon}} \leq y \leq x$
\begin{equation} \label{Dicky}
\frac{1}{C}  x \varrho(u) \leq \big|J^-(x;y)\big|   \leq C  x \varrho(u)\,,
\end{equation}
where here (and in the sequel) $u=\frac{\log x}{\log y}$.
\item For  $u\to \infty$:
 \begin{equation} \label{calcu2}
\log \varrho(u) = -u \log u \,\big(1+o(1)   \big).
\end{equation}
 \end{itemize}

\noindent We are now ready to start the

\begin{proof}[Proof of the lower estimate of $S^{\mathrm{rad}}(x)$ in Theorem~\ref{main}]
 Fix $x >2$, and choose some $2 < y \leq x$
together with some  $\ell \in \mathbb{N}$ for which  $p_\ell \le y < p_{\ell+1}$
(later it will turn out that the optimal choice for $y$ in fact is $y=e^{\frac{1}{\sqrt{2}} \sqrt{\log x \log \log x}}$).
The general strategy will be  to  apply in a first step  the Kahane-Salem-Zygmund inequality \cite[Theorem~5.3.4]{QuQu13}  in order to get
\begin{equation}  \label{step1}
\sqrt{   \frac{\left|   J^-(x;y) \right|}{y\log\log x}}  \leq K S^{\mathrm{rad}}(x)\,
\end{equation}
for some universal $K$ and then  in a  second step to optimize $y$ with analytic number theory.\\
Define   the finite Dirichlet series
\[
D_{x,y}= \sum_{\bj \in  J^-(x;y)}  \frac{1}{p_\bj^s}\,,
\]
which obviously has  length $\leq x$.
Clearly
\begin{align*}
\sum_{\bj \in  J^-(x;y)}  1 = \left| J^-(x;y) \right|\,,
\end{align*}
and therefore our aim for the proof of \eqref{step1} will be to show
\begin{equation} \label{ok}
\left\| D_{x,y}  \right\|_{\mathcal{H}_\infty^{\mathrm{rad}}} \leq K \sqrt{y  \left|   J^-(x;y) \right| \log\log x}\,.
\end{equation}
 By  Bohr's fundamental lemma (see e.g. \cite[Theorem~4.4.2]{QuQu13}) we have
\[
\left\| D_{x,y}  \right\|_{\mathcal{H}_\infty^{\mathrm{rad}}}=\int_0^1 \sup_{z \in \mathbb{T}^\ell}  \Big|\sum_{\bj \in J^-(x;y )}  r_{p_\bj}(t) z_\bj   \Big| dt\,.
\]
Hence by \eqref{length} we deduce from  the Kahane-Salem-Zygmund inequality (see the version given in
\cite[Theorem~5.3.4]{QuQu13}) that
\[
\left\| D_{x,y}  \right\|_{\mathcal{H}_\infty^{\mathrm{rad}}}
\leq K  \sqrt{\ell\,| J^-(x;y )| \, \log\log x } \,.
\]
But  trivially $\ell \leq y$ which gives \eqref{ok} and hence \eqref{step1}.
To finish the proof the number theoretical results from
\eqref{Dicky} and \eqref{calcu2} enter the game.
Assume that
$y=e^{\alpha \sqrt{\log x \log \log x}}$, where $\alpha >0$ will be specified later (as already noted it will turn out that the perfect choice is $\alpha = \frac{1}{\sqrt{2}}$). Put
\[
u:=\frac{\log x}{\log y} = \frac{1}{\alpha}  \frac{\log x}{\log \log x}\,.
\]
A simple calculation then gives
\begin{equation} \label{calcu1}
u \log u = \frac{1}{2\alpha} \sqrt{\log x\log \log x}\,\,\big(1 + o(1)\big).
\end{equation}
Note also that, taking for example $\varepsilon = 1$,  $y$  lies in the interval of validity
 of inequality  \eqref{Dicky}. Then we have:
\begin{align*}
S^{\mathrm{rad}}(x)
&
\stackrel{\eqref{step1},\eqref{Dicky}}{\geq} K_{1} \sqrt{\frac{x}{\log \log x}} \sqrt{\frac{\varrho(u)}{y}}
\\&
\stackrel{\text{ def. of $y$ }}{=} K_{1} \sqrt{\frac{x}{\log \log x}} e^{\frac{\log \varrho(u)}{2}}e^{-\frac{\alpha}{2} \sqrt{\log x  \log \log x}}
\\&
\stackrel{ \eqref{calcu2}, \eqref{calcu1}}{\geq} K_{2} \sqrt{\frac{x}{\log \log x}} e^{-\big( \frac{1}{4\alpha}+ \frac{\alpha}{2}+o(1)\big)\sqrt{\log x  \log \log x}}
= K_2 \sqrt{x} e^{-\big( \frac{1}{4\alpha}+ \frac{\alpha}{2}+o(1)\big)\sqrt{\log x  \log \log x}}
\,\,.
\end{align*}
\noindent Minimizing $ \frac{1}{4\alpha}+ \frac{\alpha}{2}$ yields the optimal parameter $\alpha = \frac{1}{\sqrt{2}}$,
and we  finally arrive at the  desired  lower estimate for $S^{\mathrm{rad}}(x)$  in Theorem~\ref{main}.
\end{proof}
Again it is possible to  graduate the result from  Theorem~\ref{main} along $m$-homogeneous polynomials.
As in \eqref{monsterdef} and \eqref{monster-p} we may define
\[
S_{p,m}(x) \,\,\, \text{  and } \,\,\, S_{p,m}^{\mathrm{rad}}(x)\,\,, \,\,\, x \in \mathbb{N}
\]
replacing $\mathcal{H}_p$ by $\mathcal{H}_{p,m}$ as well as $\mathcal{H}^{\mathrm{rad}}_p$ by $\mathcal{H}^{\mathrm{rad}}_{p,m}$, and again we see that  $S_{p,m}^{\mathrm{rad}}(x) \leq S_{p,m}(x)$.
 A careful analysis of \cite[Theorem~1.4]{BaCaQu06} and \cite[Theorem~3.1]{MaQu10} (see also
\cite{DeScSe14})   proves
\begin{align} \label{balabala}
S_{\infty,m}(x) = O \left( \frac{x^{\frac{m-1}{2m}}}{(\log x)^{m-1}} \right)\,,
\end{align}
and then  the following $m$-homogeneous variant of Theorem~\ref{main} comes naturally.

\begin{theorem} \label{we finish}
 We have, as  $x$ tends to  $\infty$
\begin{equation*}
 S_{p,m}^{\mathrm{rad}}(x)   =
\begin{cases}
O(\sqrt{x}) & \,\,\,\text{ for } \,\,\,1 \leq p < \infty \\[1ex]
 O \left(  \frac{x^{\frac{m-1}{2m}}}{(\log x)^{m-1}}\right)& \,\,\,\text{ for } \,\,\,p = \infty\,.
\end{cases}
\end{equation*}
\end{theorem}
\noindent Only the lower estimates have to be checked. For the case $1 \leq p < \infty$ argue as in the proof of Theorem~\ref{17juni}. For $p=\infty$ analyse again the proof of the lower estimate in \eqref{balabala}.

\section{Appendix: On the abscissa of a.s.-sign convergence}

One of the remarkable results of the work of Hartman \cite{Ha39} was that, unlike  the classical strips (\eqref{Bo} and \eqref{BoBoHi}), the maximal width of the two strips of the a.s.-sign convergence coincide
(\eqref{Ha1} and \eqref{Ha2}). We already pointed out \eqref{Kinfun} that this result fits in our point of view and in fact follows from our Theorem~\ref{alexanderplatz1}.\\
We wonder now what happens with the abscissas of a.s.-sign convergence and absolute convergence for vector-valued Dirichlet series. Will it again be the case that the maximal distance between these two is the same
as the maximal width for the abscissa of  a.s.-sign uniform and absolute convergence? We answer this question now. Let us introduce some notation just for this appendix; for a given Banach space $E$ we consider the numbers
\begin{gather*}
S^{\mathrm{rad}}_{c\rightarrow a} (E)
: = \sup_{D \in \mathfrak{D}(E)} \sigma_a (D)-\sigma_c^{\mathrm{rad}}(D)
\\
S^{\mathrm{rad}}_{u\rightarrow a}(E)
: = \sup_{D \in \mathfrak{D}(E)} \sigma_a (D)-\sigma_u^{\mathrm{rad}}(D)\,
\end{gather*}
By $S^{\mathrm{rad}}_{m,c\rightarrow a} (E)$ and $S^{\mathrm{rad}}_{m,u\rightarrow a} (E)$ we denote
their graduations along the homogeneity $m \in \mathbb{N}$, defined in the obvious way. Observe that $S^{\mathrm{rad}}_{u\rightarrow a}(E)$ and $S^{\mathrm{rad}}_{m,u\rightarrow a} (E)$ are just the
$S_{\infty}^{\mathrm{rad}} (E)$ and $S_{\infty, m}^{\mathrm{rad}} (E)$ that we considered before.\\
Obviously we have the trivial estimates
\begin{equation}\label{inq}
 S^{\mathrm{rad}}_{m,u\rightarrow a}(X) \,\,\leq \,\,S^{\mathrm{rad}}_{u\rightarrow a}(X)
 \quad \text{and} \quad
 S^{\mathrm{rad}}_{m,c\rightarrow a}(X) \,\,\leq \,\,S^{\mathrm{rad}}_{c\rightarrow a}(X)
 \end{equation}
as well as
\begin{equation}\label{inq3}
 S^{\mathrm{rad}}_{u\rightarrow a}(X) \,\,\leq \,\,S^{\mathrm{rad}}_{c\rightarrow a}(X)
 \quad \text{and} \quad
 S^{\mathrm{rad}}_{m,u\rightarrow a}(X) \,\,\leq \,\,S^{\mathrm{rad}}_{m,c\rightarrow a}(X)\,.
\end{equation}
Our aim is to show that for every Banach space $E$ we have
\begin{equation}\label{reichstag}
 S^{\mathrm{rad}}_{u\rightarrow a}(E) =S^{\mathrm{rad}}_{c\rightarrow a}(E)
 =
S^{\mathrm{rad}}_{m,c\rightarrow a}(E)=1 - \frac{1}{\cot (E)}\,;
\end{equation}
and if $E$ is infinite-dimensional, then we can also put $S^{\mathrm{rad}}_{m,u\rightarrow a}(E)$ within the previous inequalities. The equalities for  $ S^{\mathrm{rad}}_{u\rightarrow a}(E)$ and
$S^{\mathrm{rad}}_{m,u\rightarrow a}(E)$ follow from Theorem~\ref{alexanderplatz1} and Proposition~\ref{contrast} with $p= \infty$.
We once again mention that the  scalar case $E=\mathbb{C}$ is due to Bohr, Bohnenblust-Hille and Hartman.
\\
We consider again the space
\[
\Rad(E) := \Big\{ a= (a_n) \in E^{\mathbb{N}}\,\,: \,\, \sum_{n=1}^\infty a_n r_n \in L_1([0,1];E)\Big\}
\]
which together with the norm
\[
\| (a_n)_{n} \|_{\Rad(E)} := \int_0^1 \Big\|  \sum_{n=1}^\infty a_n r_n(t) \Big\|_E dt
\]
forms a Banach space. Recall that $(a_n)_{n}$ belongs to $\Rad(E)$ if and only if $\sum_{n=1}^\infty a_n \varepsilon_n$
converges for almost all choices of signs $\varepsilon_n$.
In particular,
\begin{equation} \label{reform}
\sigma_c^{\mathrm{rad}} (D)  = \inf \Big\{ \sigma \in \mathbb{R}\colon \textstyle
\big(\frac{a_n}{n^\sigma}\big) \in \Rad(E) \Big\} \,.
\end{equation}

Let us note that the key ingredient to get descriptions of the width of the strip in the spirit of Maurizi-Queff\'elec (see \eqref{leipziger}, \eqref{tegel}, \eqref{tegel1} and \eqref{tegel2}) is to have a norm that provides a
proper control of the size of the partial sums, like in \eqref{gendarmenmarkt}. Observe that now, by Kahane's contraction principle, we have that  for each $N$
\[
\| (a_n)_{n=1}^N \|_{\Rad(E)}  \leq \big\| (a_n) \big\|_{\Rad(E)}\,.
\]
Proceeding as in Proposition~\ref{oranienburg}, using this instead of Proposition~\ref{unterdenlinden}, we obtain
\begin{gather}
S^{\mathrm{rad}}_{c\rightarrow a}(E) = \inf \bigg\{ \sigma >0 \, \big|\,\exists c_{\sigma}\,
\forall
D=  \sum_{n=1}^{N} a_{n} n^{-s} \in \mathfrak{D}(E)
: \,\sum_{n=1}^{N} \Vert a_{n} \Vert \leq c_{\sigma} N^{\sigma} \big\| (a_n)_{n=1}^N \big\|_{\Rad(E)} \bigg\} \label{hauptbhf1} \\
S^{\mathrm{rad}}_{m,c\rightarrow a}(E) = \inf \bigg\{ \sigma >0 \, \big|\,\exists c_{\sigma}\,
\forall
D=  \sum_{n=1}^{N} a_{n} n^{-s} \in \mathfrak{D}_m(E)
: \,\sum_{n=1}^{N} \Vert a_{n} \Vert \leq c_{\sigma} N^{\sigma} \big\| (a_n)_{n=1}^N \big\|_{\Rad(E)} \bigg\} \label{hauptbhf2}
\end{gather}
Note that in the scalar case $E=\mathbb{C}$, by Khinchine's inequality, we see that
\[
S^{\mathrm{rad}}_{c\rightarrow a}(\mathbb{C}) = \inf \bigg\{ \sigma >0 \, \big|\,\exists c_{\sigma}\,
\forall
D=  \sum_{n=1}^{N} a_{n} n^{-s} \in \mathfrak{D}(\mathbb{C})
: \,\sum_{n=1}^{N}  | a_{n}| \leq c_{\sigma} N^{\sigma} \Big(\sum_{n=1}^N |a_n|^2\Big)^{\frac{1}{2}} \bigg\} \,,
\]
and hence, applying the  Cauchy-Schwarz  inequality, we obtain  Hartman's result $S^{\mathrm{rad}}_{c\rightarrow a}(\mathbb{C}) \, = \, \frac{1}{2}$.

\noindent Finally to complete the proof of  \eqref{reichstag} it only remains to show the following

\begin{theorem}\label{hart2}
For every Banach space $E$ and every $m \in \mathbb{N}$
\[
S^{\mathrm{rad}}_{c\rightarrow a}(E)= S^{\mathrm{rad}}_{m,c\rightarrow a}(E)= 1 - \frac{1}{\cot(E)}\,.
\]
\end{theorem}

\begin{proof}
We begin with the equality for $S^{\mathrm{rad}}_{c\rightarrow a}(E)$. By \eqref{inq} and the lower estimate for $S^{\mathrm{rad}}_{u\rightarrow a}(E)$ from \eqref{reichstag} we have to check
\begin{equation}\label{hart3}
 S^{\mathrm{rad}}_{c\rightarrow a} (E) \leq 1 - \frac{1}{\cot(E)}\,.
\end{equation}
Take  $q > \cot(E)$. Then for each finite Dirichlet series $D= \sum_{n=1}^N a_n \frac{1}{n^{s}} \in \mathfrak{D}(E)$ by H\"older's inequality
\begin{align*}
\sum_{n=1}^N \|a_n\|
\leq
N^{\frac{1}{q'}} \Big(\sum_{n=1}^N \|a_n\|^q\Big)^{\frac{1}{q}}
\leq
C_{q}(E) N^{\frac{1}{q'}}\big\| (a_n)\big\|_{\Rad(E)}\,.
\end{align*}
Hence we obtain from \eqref{hauptbhf1} that $S^{\mathrm{rad}}_{c\rightarrow a} \leq 1 - \frac{1}{q}$, the
conclusion.\\
We finish by giving the argument  for $S^{\mathrm{rad}}_{m,c\rightarrow a}(E)$. If $E$ is infinite-dimensional, by
\eqref{inq} and \eqref{inq3}, the result for $S^{\mathrm{rad}}_{m,u\rightarrow a}(E)$ from \eqref{reichstag}, and \eqref{hart3} we have
\[
1 - \frac{1}{\cot(E)} = S^{\mathrm{rad}}_{m,u\rightarrow a}(E)
\leq S^{\mathrm{rad}}_{m,c\rightarrow a}(E)
\leq
S^{\mathrm{rad}}_{c\rightarrow a}(E) = 1 - \frac{1}{\cot(E)}\,.
\]
On the other hand, if $E$ is finite dimensional we can argue as in \eqref{pergamon} to show that $S^{\mathrm{rad}}_{1,c\rightarrow a}(\mathbb{C}) \geq 1/2$. This completes the proof.
\end{proof}

\noindent  In  the scalar case $E= \mathbb{C}$
and in view of Proposition \ref{reform}, it is again  possible to graduate $S^{\mathrm{rad}}_{c\rightarrow a}(\mathbb{C})$ and  $S^{\mathrm{rad}}_{m,c\rightarrow a}(\mathbb{C})$, respectively,  along
the length of the Dirichlet polynomials. As in \eqref{monsterdef}, for $x \geq 1$ we define
\[
S^{\mathrm{rad}}_{c\rightarrow a}(x)  :=   \sup_{(a_{n})_{n \in \mathbb{N}} \subseteq \mathbb{C}} \frac{\sum_{n \leq x} a_n|}{\big\|(a_n)_{n \leq x} \big\|_{\Rad(E)}}\,,
\]
and similarly $S^{\mathrm{rad}}_{m,c\rightarrow a}(x)$. Then by Khinchine's inequality and the Cauchy-Schwarz inequality we obviously have
\[
S_{c \rightarrow a}^{\mathrm{rad}}(x)  = S_{m, c \rightarrow a}^{\mathrm{rad}}(x) \sim \sqrt{x}\,.
\]


\begin{thebibliography}{10}

\bibitem{AlOlSa12}
A.~Aleman, J.-F. Olsen, and E.~Saksman.
\newblock Fourier multipliers for {H}ardy spaces of {D}irichlet series.
\newblock {\em arXiv:1210.4292}, 2012.

\bibitem{BaCaQu06}
R.~Balasubramanian, B.~Calado, and H.~Queff{\'e}lec.
\newblock The {B}ohr inequality for ordinary {D}irichlet series.
\newblock {\em Studia Math.}, 175(3):285--304, 2006.

\bibitem{Ba02}
F.~Bayart.
\newblock Hardy spaces of {D}irichlet series and their composition operators.
\newblock {\em Monatsh. Math.}, 136(3):203--236, 2002.

\bibitem{BaDeFrMaSe14}
F.~Bayart, A.~Defant, L.~Frerick, M.~Maestre, and P.~Sevilla-Peris.
\newblock Monomial series expansion of $H_p$-functions and multipliers of
  $\mathcal{H}_p$-Dirichlet series.
\newblock {\em preprint}, 2014.

\bibitem{BaPeSe14}
F.~Bayart, D.~Pellegrino, and J.~B. Seoane-Sep{\'u}lveda.
\newblock The {B}ohr radius of the {$n$}-dimensional polydisk is equivalent to
  {$\sqrt{(\log n)/n}$}.
\newblock {\em Adv. Math.}, 264:726--746, 2014.

\bibitem{BaQuSe15}
F.~Bayart, H.~Queff\'elec, and K.~Seip.
\newblock Approximation numbers of composition operators on ${H}_{p}$ spaces of
  {D}irichlet series.
\newblock {\em arXiv:1406.0445}, 2014.

\bibitem{BoHi31}
H.~F. Bohnenblust and E.~Hille.
\newblock On the absolute convergence of {D}irichlet series.
\newblock {\em Ann. of Math. (2)}, 32(3):600--622, 1931.

\bibitem{Bo13_Goett}
H.~Bohr.
\newblock {\"U}ber die {B}edeutung der {P}otenzreihen unendlich vieler
  {V}ariablen in der {T}heorie der \textit{{D}irichlet}--schen {R}eihen
  $\sum\,\frac{a_n}{n^s}$.
\newblock {\em Nachr. Ges. Wiss. G{\"o}ttingen, Math. Phys. Kl.}, pages
  441--488, 1913.

\bibitem{CaDeSe14}
D.~Carando, A.~Defant, and P.~Sevilla-Peris.
\newblock Bohr's absolute convergence problem for {${\cal H}_p$}-{D}irichlet
  series in {B}anach spaces.
\newblock {\em Anal. PDE}, 7(2):513--527, 2014.

\bibitem{CoGa86}
B.~J. Cole and T.~W. Gamelin.
\newblock Representing measures and {H}ardy spaces for the infinite polydisk
  algebra.
\newblock {\em Proc. London Math. Soc. (3)}, 53(1):112--142, 1986.

\bibitem{Br08}
R.~de~la Bret{\`e}che.
\newblock Sur l'ordre de grandeur des polyn\^omes de {D}irichlet.
\newblock {\em Acta Arith.}, 134(2):141--148, 2008.

\bibitem{DeFrOrOuSe11}
A.~Defant, L.~Frerick, J.~Ortega-Cerd{\`a}, M.~Ouna{\"{\i}}es, and K.~Seip.
\newblock The {B}ohnenblust-{H}ille inequality for homogeneous polynomials is
  hypercontractive.
\newblock {\em Ann. of Math. (2)}, 174(1):485--497, 2011.

\bibitem{DeGaMaPG08}
A.~Defant, D.~Garc{\'i}a, M.~Maestre, and D.~P{\'e}rez-Garc{\'i}a.
\newblock Bohr's strip for vector valued {D}irichlet series.
\newblock {\em Math. Ann.}, 342(3):533--555, 2008.

\bibitem{DeGaMaSe14}
A.~Defant, D.~Garc{\'{\i}}a, M.~Maestre, and P.~Sevilla-Peris.
\newblock Bohr's fundamental theorem for vector-valued {H}ardy-{D}irichlet
  spaces.
\newblock {\em preprint}, 2014.

\bibitem{DeMaSc12}
A.~Defant, M.~Maestre, and U.~Schwarting.
\newblock Bohr radii of vector valued holomorphic functions.
\newblock {\em Adv. Math.}, 231(5):2837--2857, 2012.

\bibitem{DePe}
A.~Defant and A.~P\'erez.
\newblock The vector-valued {B}ohr transform.
\newblock {\em preprint}, 2014.

\bibitem{DeScSe14}
A.~Defant, U.~Schwarting, and P.~Sevilla-Peris.
\newblock Estimates for vector valued {D}irichlet polynomials.
\newblock {\em Monatsh. Math.}, 175(1):89--116, 2014.

\bibitem{DiJaTo95}
J.~Diestel, H.~Jarchow, and A.~Tonge.
\newblock {\em Absolutely summing operators}, volume~43 of {\em Cambridge
  Studies in Advanced Mathematics}.
\newblock Cambridge University Press, Cambridge, 1995.

\bibitem{Ha39}
P.~Hartman.
\newblock On {D}irichlet series involving random coefficients.
\newblock {\em Amer. J. Math.}, 61:955--964, 1939.

\bibitem{HeLiSe97}
H.~Hedenmalm, P.~Lindqvist, and K.~Seip.
\newblock A {H}ilbert space of {D}irichlet series and systems of dilated
  functions in {$L^2(0,1)$}.
\newblock {\em Duke Math. J.}, 86(1):1--37, 1997.

\bibitem{HiTe93} A. Hildebrand, and G. Tenenbaum. \newblock Integers without large prime factors. \newblock {\em J. Théor. Nombres Bordeaux},
5 (2):411--484, 1993.


\bibitem{KoQu01}
S.~V. Konyagin and H.~Queff{\'e}lec.
\newblock The translation {$\frac12$} in the theory of {D}irichlet series.
\newblock {\em Real Anal. Exchange}, 27(1):155--175, 2001/02.

\bibitem{LiTz79}
J.~Lindenstrauss and L.~Tzafriri.
\newblock {\em Classical {B}anach spaces. {II}}, volume~97 of {\em Ergebnisse
  der Mathematik und ihrer Grenzgebiete [Results in Mathematics and Related
  Areas]}.
\newblock Springer-Verlag, Berlin, 1979.
\newblock Function spaces.

\bibitem{MaQu10}
B.~Maurizi and H.~Queff{\'e}lec.
\newblock Some remarks on the algebra of bounded {D}irichlet series.
\newblock {\em J. Fourier Anal. Appl.}, 16(5):676--692, 2010.

\bibitem{Qu95}
H.~Queff{\'e}lec.
\newblock H. {B}ohr's vision of ordinary {D}irichlet series; old and new
  results.
\newblock {\em J. Anal.}, 3:43--60, 1995.

\bibitem{QuQu13}
H.~Queff{\'e}lec and M.~Queff{\'e}lec.
\newblock {\em Diophantine approximation and {D}irichlet series}, volume~2 of
  {\em Harish-Chandra Research Institute Lecture Notes}.
\newblock Hindustan Book Agency, New Delhi, 2013.

\bibitem{Te95}
G.~Tenenbaum.
\newblock {\em Introduction to analytic and probabilistic number theory},
  volume~46 of {\em Cambridge Studies in Advanced Mathematics}.
\newblock Cambridge University Press, Cambridge, 1995.
\newblock Translated from the second French edition (1995) by C. B. Thomas.

\end{thebibliography}
\end{document}